\documentclass{amsart}
\usepackage[utf8]{inputenc}
\usepackage{amsmath, amsthm, amssymb, amsfonts, mathtools, bbold, stmaryrd, latexsym, enumerate}    
\usepackage{wasysym}
\usepackage{dsfont}

\pdfoutput=1

\usepackage[dvipsnames]{xcolor}
\usepackage{stackengine}
\usepackage{float}
\usepackage{enumitem}
\usepackage{tikz}
\usepackage{comment}
\usepackage[linesnumbered,ruled,commentsnumbered]{algorithm2e}

\usepackage{MnSymbol} 
\usepackage{graphicx}
\usepackage{caption}
\usepackage[all]{xypic}
\usepackage{verbatim}
\usepackage{chemfig, chemnum}
\usepackage{subcaption}

\usepackage[colorlinks]{hyperref}
\hypersetup{
    colorlinks,
    linkcolor={BrickRed},
    citecolor={ForestGreen}
}
\usepackage[nameinlink]{cleveref}

\usepackage[parfill]{parskip}
\usepackage[margin=1in]{geometry}

\theoremstyle{plain}
\newtheorem{theorem}{Theorem}
\numberwithin{theorem}{section}

\newtheorem{lemma}[theorem]{Lemma}

\newtheorem{definition}[theorem]{Definition}

\theoremstyle{definition}
\newtheorem{remark}[theorem]{Remark}

\newtheorem{conjecture}[theorem]{Conjecture}

\newcommand{\R}{\mathbb{R}}

\newcommand{\CC}{\mathbb{C}}
\newcommand{\ZZ}{\mathbb{Z}}

\newcommand{\G}{\text{G}}
\newcommand{\Trop}{\text{Trop}}
\newcommand{\Dr}{\text{Dr}}

\title{The Chirotropical Grassmannian}

\author{Dario Antolini}
\address{%
	Dario Antolini\newline
	Dipartimento di Matematica, Università di Trento\newline
	Email: \href{mailto:dario.antolini-1@unitn.it}{dario.antolini-1@unitn.it}
}

\author{Nick Early}
\address{%
	Nick Early\newline
	Institute for Advanced Study, Princeton NJ\newline
	Email: \href{mailto:earlnick@ias.edu}{earlnick@ias.edu}
}
	
\begin{document}

\maketitle

\begin{abstract}
		\noindent
        Recent developments in particle physics have revealed deep connections between scattering amplitudes and tropical geometry. From the heart of this relationship emerged the chirotropical Grassmannian $\text{Trop}^\chi \G(k,n)$ and the chirotropical Dressian $\text{Dr}^\chi(k,n)$, polyhedral fans built from uniform realizable chirotopes that encode the combinatorial structure of Generalized Feynman Diagrams. We prove that $\text{Trop}^\chi \G(3,n) = \text{Dr}^\chi(3,n)$ for $n = 6,7,8$, and develop algorithms to compute these objects from their rays modulo lineality. Using these algorithms, we compute all chirotropical Grassmannians $\text{Trop}^\chi \G(3,n)$ for $n = 6,7,8$ across all isomorphism classes of chirotopes. We prove that each chirotopal configuration space $X^\chi(3,6)$ is diffeomorphic to a polytope and propose an associated canonical logarithmic differential form. Finally, we show that the equality between chirotropical Grassmannian and Dressian fails for $(k,n) = (4,8)$.
        
	\end{abstract}
	
	\section{Introduction}
        The interplay between tropical geometry and particle physics has recently unveiled fascinating connections, particularly in the study of scattering amplitudes. At the heart of this relationship lies the chirotropical Grassmannian $\text{Trop}^\chi \G(k,n)$ and the chirotropical Dressian $\text{Dr}^\chi(k,n)$, introduced by Cachazo, Zhang and one of the authors \cite{CEZ2024A2} in the context of generalized bi-adjoint scalar amplitudes.
        
        In tropical geometry, the connection starts with the introduction of the complex tropical Grassmannian $\Trop \ \G(k,n)$ by Speyer and Sturmfels \cite{speyersturmfels2004}.  In 2019, in the context of theoretical particle physics, Cachazo, Guevara, Mizera and one of the authors \cite{CEGM2019} observed an equivalence between $\Trop \ \G(2,n)$, the space of phylogenetic trees on $n$ leaves, with Feynman diagrams on $n$ massless particles in the $\phi^3$ theory.  This link enables the computation of bi-adjoint scalar amplitudes at tree level in $\phi^3$ theory using the maximal cones of the tropical Grassmannian $\Trop \ \G(2,n)$.  In particular, $\Trop \ \G(2,n)$ contains a subfan called the positive tropical Grassmannian, which is the building block for biadjoint scalar amplitudes.  On the other hand, the positive tropical Grassmannian $\Trop^+ \G(k,n)$ was introduced for general $(k,n)$ by Speyer and Williams in \cite{speyerwilliams2004}, and was shown to enjoy various remarkable properties.  Namely, it is determined by only the 3-term tropical Plücker relations \cite{speyerwilliamspos2020, AHL2020}, using that every positively oriented matroid is realizable \cite{ardilarinconwilliams}.  Moreover, $\text{Trop} \ \G(2,n)$ admits a uniform, $2^{n-3}$-fold covering with relabelings of $\text{Trop}^+\G(2,n)$.  That is, every maximal cone appears in exactly $2^{n-3}$ chirotropical Grassmannians, as one can see by counting automorphisms of cubic trees. This leads a partial decomposition of the bi-adjoint scalar amplitudes into partial amplitudes corresponding only to relabelings of the positive tropical Grassmannian $\Trop^+ \G(2,n)$.

	In recent work in theoretical physics \cite{CEZ2024A2,CEZ2024A4B}, Cachazo, Zhang and one of the authors pursued a deeper exploration of the \textit{real} tropical Grassmannian $\Trop^\R \G(k,n)$, motivated by the need to find a uniform covering of it that would be analogous to the covering of $\Trop^\R \G(2,n)$ with $2^{n-3}$ copies of the positive part; but this particular feature does not have an obvious answer when $k\ge 3$.  Including only relabelings of $\Trop^+ \G(k,n)$ leaves holes and more exotic pieces are needed to fill in the gaps.  It is the full covering that should reveal the deepest new physics in the story initiated in \cite{CEGM2019}.
    
    This led to the introduction of the {\em chirotropical Grassmannian} $\Trop^\chi \G(k,n)$, one for each realizable, uniform chirotope (or oriented matroid) $\chi$ of rank $k$. These polyhedral fans were shown to encode combinatorially certain Grothendieck residues arising in a formula in \cite{CEGM2019}. This gives rise to generalized biadjoint scalar amplitudes, which are computed directly from higher tropical Grassmannians $\Trop \ \G(k,n)$. In particular, it is shown that, in the case $k = 3$ and $n = 6, 7, 8$, the CEGM formula is equal to the Laplace transform of the chirotropical Grassmannian $\Trop^\chi \G(k,n)$.

    Motivated by these works, in this paper we address various questions regarding chirotropical Grassmannians. These include, in particular, their parameterization, their computation and their realizability. 
    
    The material is organized as follows. In \Cref{sec:prelim}, we give main definitions and formulate the theoretical framework for our computations of the chirotropical Dressian.  In \Cref{sec:realizability}, we characterize the maximal cones of the Dressian (\Cref{lem: full dressian cliques}) and of the chirotropical Dressian (\Cref{lem: full chirotropical dressian cliques}), determined by compatible and $\chi$-compatible rays respectively. 
    
    In \Cref{sec:computing}, we develop an efficient algorithm (\Cref{algorithm-dr}) to compute the Dressian $\text{Dr}(k,n)$ from its rays modulo lineality. Given this collection of rays, our algorithm can be used to compute all chirotropical Dressians $\text{Dr}^\chi(k,n)$ (\Cref{algorithm-chi}). \Cref{sec:3678} contains the proof of our main realizability result. This confirms a conjecture in \cite{CEZ2024A2}.
    
    \begin{theorem}
    	\label{thm:realizability-intro}
    	For $n=6,7,8$ and any uniform realizable chirotope $\chi \in \{\pm1\}^{\binom{n}{3}}$ of rank 3, the chirotropical Dressian is realizable: we have the equality of sets
    	\[
    	\Dr^\chi(3,n) = \Trop^\chi \G(3,n), \qquad n = 6, 7, 8.
    	\]
    \end{theorem}
    
    In the same section, we present the results of our computations of all of the chirotropical Grassmannians $\Trop^\chi \G(3,n)$ for $n = 6, 7, 8$ modulo lineality using \Cref{algorithm-chi}. In \Cref{sec:36polytopal}, we show that every $\Trop^\chi \G(3,6)$ is the normal fan of a polytope. We introduce a Zenodo page \cite{datachirotropicalization20241} with all the data from our computation and we explain how to use it in \Cref{sec:code}.  Finally, in \Cref{sec:futurework}, we exhibit a pair $(\chi,\pi)$ consisting of a chirotope $\chi \in \{\pm1\}^{\binom{8}{4}}$ and a non-realizable chirotropical Plücker vector $\pi \in \Dr^\chi(4,8)$. This proves that the chirotropical Dressian $\Dr^\chi(4,8)$ is strictly bigger than the chirotropical Grassmannian $\Trop^\chi \G(4,8)$.
		
	\section{Preliminaries}
        \label{sec:prelim}
	
	Let $\lbrack n\rbrack = \{1,\ldots, n\}$ and let $\binom{\lbrack n\rbrack}{k}$ be the collection of $k$-element subsets of $\lbrack n\rbrack$.
    
    Throughout this work, we denote by $\G(k,n)$ the complex Grassmannian. We consider its Plücker embedding as a subvariety of $\mathbb{P}^{\binom{n}{k} - 1}$ and we denote by $\Trop \ \G(k,n)$ its tropicalization in this embedding using the trivial valuation. The open Grassmannian $\G^\circ(k,n)$ is defined by:
	\[
	\G^\circ(k,n) = \Biggl\{ g \in \G(k,n) : \prod_{I \in \binom{[n]}{k}} p_I(g) \neq 0 \Biggr\},
	\]
	where $p_I(g)$ denotes the Plücker coordinate of $g$ of index $I \in \binom{[n]}{k}$. The Grassmannian $\G(k,n)$ comes with a natural right action by the $n$-dimensional complex torus $(\CC^\times)^n$. Quotienting by this action, we obtain
	\[
	X(k,n) = \G^\circ(k,n) / (\CC^{\times})^n
	\]
	which is the moduli space of $n$ distinct points in $\mathbb{P}^{k-1}$, modulo projective transformations. Its tropicalization is $\Trop \ X(k,n) = \Trop \ \G(k,n) / L_{k,n}$, where $L_{k,n}$ is the lineality space of $\Trop \ \G(k,n)$; it is obtained as the image of the linear map with coordinates $\pi_I= \sum_{i \in I} x_i$ for $x\in \mathbb{R}^n$.
    $\Trop \ X(k,n)$ is a pure $(k-1)(n-k-1)$-dimensional polyhedral fan.
	
	The real Grassmannian, denoted by $\G_\R(k,n)$, and its open part, denoted by $\G_\R^\circ(k,n)$, play also a key role in the paper. They provide the main ingredient for the chirotropical Grassmannian: realizable uniform chirotopes. These are also known in the literature as realizable uniform oriented matroids.  Modding out by the $n$-torus, they are realized by generic arrangements of $n$ hyperplanes in real projective space $\mathbb{P}^{k-1}$. Each of these arrangements determines a torus orbit in $\G_\R^\circ(k,n)$. The vector of signs of the Plücker coordinates of any point in the orbit gives the chirotope representation of the arrangement. For more information on oriented matroids and their chirotope representation, see \cite{bjorner}.
	
    In this paper, we start directly with the chirotope representation. We work with \textit{realizable}, \textit{uniform} chirotopes of rank $k$, that is, vectors of signs $\chi \in \{\pm1\}^{\binom{n}{k}}$ of Plücker coordinates of a point $g \in \G^\circ_\R(k,n)$:
    \[
        \chi_I = \text{sign} \ p_I(g) \qquad \forall I \in \binom{[n]}{k}.
    \]
    Two chirotopes are isomorphic if they are related by relabeling of the ground set $\lbrack n\rbrack$ via a permutation $\sigma \in S_n$:
\[
\chi_I \longmapsto \chi_{\sigma(I)}
\]
or by reorientation according to the action of a torus element $ t \in \{\pm1\}^n$:
\[
\chi_I \longmapsto \Biggl(\prod_{i\in I}t_i\Biggr) \chi_I.
\]
        We further denote by
        \[
            X^\chi(k,n) = \left\{g\in G_\mathbb{R}(k,n): \text{sign} \ p_J(g) = \chi_J \ \forall J \in \binom{[n]}{k} \right\} \big\slash (\mathbb{R}_{>0})^n
        \]
        the \textit{chirotopal configuration space}.  Note that due to the torus action, more than one chirotope determines the same $X^\chi(k,n)$.

        This allows us to classify chirotopes into isomorphism classes.  We will need the following enumeration of isomorphism classes of rank 3 chirotopes on 6, 7 and 8 elements. The references for these results are scattered over different articles, but they are organized systematically in \cite[Section 6.5]{Finschi2001}.
	
	\begin{theorem}
		There are 4, 11 and 135 isomorphism classes of rank 3 uniform realizable chirotopes on $n = 6, 7, 8$ elements respectively.
	\end{theorem}
	
	We remark that \cite[Chapter 6]{Finschi2001} explains how to compute these isomorphism classes. The data is available in an online catalog \cite{FinschiCatalog}.
	
	\begin{definition}[Chirotropical hypersurface]
		Let $\chi \in \{-1,+1\}^{\binom{n}{k}}$ be a chirotope. Let $g$ be a polynomial with real coefficients in the $\binom{n}{k}$ Plücker
  variables. Denote by $\text{supp}(g)$ the set of monomials appearing in $g$.  Let
		\begin{equation*}
			g^\chi_+ = \sum_{\substack{m \ \in \ \text{supp}(g) \\ \text{sign}(m(\chi)) = 1}} m, \qquad g^\chi_- = \sum_{\substack{m \ \in \ \text{supp}(g) \\ \text{sign}(m(\chi)) = -1}} m.
		\end{equation*}
		Denote by $V(g) \subset \mathbb{P}^{\binom{n}{k}-1}$ the hypersurface defined by the polynomial $g$.  The \emph{chirotropical hypersurface} $\Trop^\chi V(g)$ is the set of all points where the minimum in $\Trop(g)(x)$ is attained at least twice and at least once in both $\Trop(g^\chi_+)(x)$ and $\Trop(g^\chi_{-})(x)$.  More formally,
		\begin{equation}
			\label{eqn:min+min-}
			\Trop^\chi V(g) = \Big\{ x \in \R^{\binom{n}{k}} \ \Big \vert \ \Trop(g^\chi_+)(x) = \Trop(g^\chi_{-})(x) \Big\}.
		\end{equation}
	\end{definition}

	\begin{definition}[Dressian and chirotropical Dressian \cite{CEZ2024A2}]
	Let $n, k$ be integers, $1 \leq k \leq n$. A vector $\pi \in \R^{\binom{n}{k}}$ is said to satisfy the \emph{tropical Plücker relations} if the minimum value of
    \begin{equation}
        \{\pi_{Lab} + \pi_{Lcd}, \pi_{Lac} + \pi_{Lbd}, \pi_{Lad}+  \pi_{Lbc}\}    
    \end{equation}
    is achieved at least twice, for all $L \in \binom{[n]}{k-2}$ and $\{a,b,c,d\} \in \binom{[n]}{4}$ such that $L \cap \{a, b, c, d\} = \emptyset$. Such a vector is called a \emph{tropical Plücker vector}. The Dressian $\Dr(k,n)$ is the set of all tropical Plücker vectors $\pi \in \R^{\binom{n}{k}}$. Equivalently, let $P_{k,n}$ be the collection of the 3-term Plücker relations of $\G(k,n)$. Then:
        \begin{equation}
            \Dr(k,n) = \bigcap_{g \in P_{k,n}} \Trop \ V(g).   
        \end{equation}
        Let $\chi \in \{\pm1\}^{\binom{n}{k}}$ be a uniform realizable chirotope. An element $\pi \in \mathbb{R}^{\binom{n}{k}}$ is said to satisfy the $\chi$\emph{-tropical Plücker relations} provided that, whenever
		\begin{equation}\label{eq: chirotope relations}
			\chi_{Lab}\chi_{Lcd} = \chi_{Lad}\chi_{Lbc} = \chi_{Lac}\chi_{Lbd}
		\end{equation}	
		we have
		\begin{equation}
			\label{eqn:chirotropicalpluecker}
			\pi_{Lac} + \pi_{Lbd} = \min\{\pi_{Lab} + \pi_{Lcd},\pi_{Lad} + \pi_{Lbc}\}
		\end{equation}
        for all $L \in \binom{[n]}{k-2}$ and $\{a,b,c,d\} \in \binom{[n]}{4}$ such that $L \cap \{a, b, c, d\} = \emptyset$.
		Such an element $\pi \in \R^{\binom{n}{k}}$ is called a $\chi$\emph{-tropical Plücker vector}. The set of all such vectors is the \emph{chirotropical Dressian}, denoted by:
		\[
		\Dr^\chi(k,n) = \Bigl\{ \pi \in \mathbb{R}^{\binom{n}{k}} : \pi \text{ is a } \chi\text{-tropical Plücker vector}\Bigr\}.
		\]
	Equivalently, with the same notations as above:
		\begin{equation}
			\label{eqn:pluecker}
			\Dr^\chi(k,n) = \bigcap_{g \in P_{k,n}} \Trop^\chi V(g),
		\end{equation}
	\end{definition}
		
	\begin{definition}[Chirotropical Grassmannian and moduli space \cite{CEZ2024A2}]
		 Let $n,k$ be integers, $1\leq k \leq n$ and let $\chi \in \{\pm1\}^{\binom{n}{k}}$ be a uniform realizable chirotope. Denote by $I_{k,n} \subset \ZZ\Bigl[p_I : I \in \binom{[n]}{k}\Bigr]$ the Plücker ideal of the Grassmannian $\G(k,n)$. The \emph{chirotropical Grassmannian} is defined by:
		\begin{equation}
			\Trop^\chi \G(k,n) = \bigcap_{g \in I_{k,n}} \Trop^\chi V(g).
		\end{equation}
		Recall that $L_{k,n}$ denotes the lineality space of $\Trop \ \G(k,n)$. We define the \emph{chirotropical moduli space} as:
		\[
		\Trop^\chi X(k,n) = \Trop^\chi \G(k,n) / L_{k,n}.
		\]
	\end{definition}

    \begin{remark}
    \label{note:+}
        If we consider the totally positive chirotope $+ = (1, \ldots, 1) \in \binom{n}{k}$, then the chirotropical Grassmannian $\Trop^+ \G(k,n)$ is exactly the tropical totally positive Grassmannian as defined in \cite{speyerwilliams2004}.
    \end{remark}

 \section{Maximal cones of Dressians and Chirotropical Dressians}
 \label{sec:realizability}

In this section, we show that the coarsest fan structure on the Dressian, as inherited from subdivisions of $\Delta_{k,n}$, is completely characterized by its rays.  In particular, the result holds for the chirotropical Dressian. 
 This provides the main technical tool that makes \Cref{algorithm-dr} and \Cref{algorithm-chi} work. We begin with the following lemma.

	\begin{lemma}
		\label{lemma:computationfulldr}
		Suppose that $\Pi = \{\pi_1,\ldots, \pi_d\} \subset \Dr(k,n)$ is a collection of tropical Plücker vectors such that each sum $\pi_i + \pi_j$ is again a tropical Plücker vector, that is, $\pi_i + \pi_j \in \Dr(k,n)$ for all $i,j \in [d]$. Then the cone
		\[
			\text{cone}(\Pi) = \{ c_1 \pi_1 + \cdots + c_d \pi_d : c_1, \ldots, c_d \geq 0\}
		\]
		is contained in the Dressian $\Dr(k,n)$.
	\end{lemma}
	
	In our proof, we rely on the bijection between (relative interiors of) cones in the Dressian $\Dr(k,n)$ and regular matroid subdivisions of the hypersimplex
	\begin{eqnarray}
		\Delta_{k,n} & = & \left\{x\in \lbrack 0,1\rbrack^n: \sum_{j=1}^n x_j =k \right\}
	\end{eqnarray}
	with vertices $\sum_{j \in J} e_j$ for any $k$-subset $J \subset [n]$.  See, for example, \cite{GGMS} and \cite{sturmfelsmaclagan2015}. Moreover, we recall that each hypersimplex $\Delta_{k,n}$ has octahedral faces $F_{L; \{a,b,c,d\}}$ for each $L \in \binom{[n]}{k-2}$ and $\{a, b, c, d\} \in \binom{[n]}{4}$ such that $L \cap \{a,b,c,d\} = \emptyset$.

    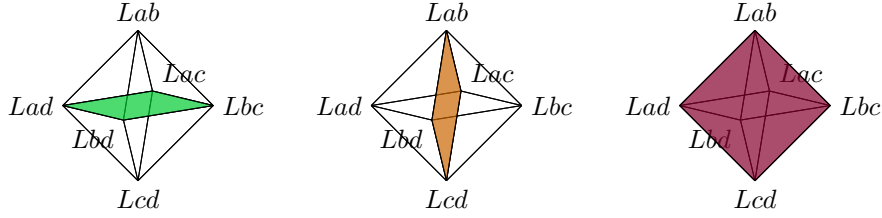
\begin{figure}
    \centering
    \begin{tikzpicture}[
    line join=bevel,
    z=-5.5, 
    ]
\coordinate (A1) at (0,0,-1);
\node at (A1) [above right] {$Lac$}; 
\coordinate (A2) at (-1,0,0);
\node at (A2) [left] {$Lad$}; 
\coordinate (A3) at (0,0,1);
\node at (A3) [below left] {$Lbd$}; 
\coordinate (A4) at (1,0,0);
\node at (A4) [right] {$Lbc$};
\coordinate (B1) at (0,1,0);
\node at (B1) [above] {$Lab$};
\coordinate (C1) at (0,-1,0);
\node at (C1) [below] {$Lcd$};

\draw (A1) -- (A2) -- (B1) -- cycle;
\draw (A4) -- (A1) -- (B1) -- cycle;
\draw (A1) -- (A2) -- (C1) -- cycle;
\draw (A4) -- (A1) -- (C1) -- cycle;
\draw (A2) -- (A3) -- (B1) -- cycle;
\draw (A3) -- (A4) -- (B1) -- cycle;
\draw (A2) -- (A3) -- (C1) -- cycle;
\draw (A3) -- (A4) -- (C1) -- cycle;
\draw [fill opacity=0.7,fill=green!80!blue] (A1) -- (A2) -- (A3) -- (A4) -- cycle;
\end{tikzpicture} \quad 
        \begin{tikzpicture}[
        line join=bevel,
        z=-5.5, 
        ]
\coordinate (A1) at (0,0,-1);
\node at (A1) [above right] {$Lac$}; 
\coordinate (A2) at (-1,0,0);
\node at (A2) [left] {$Lad$}; 
\coordinate (A3) at (0,0,1);
\node at (A3) [below left] {$Lbd$}; 
\coordinate (A4) at (1,0,0);
\node at (A4) [right] {$Lbc$};
\coordinate (B1) at (0,1,0);
\node at (B1) [above] {$Lab$};
\coordinate (C1) at (0,-1,0);
\node at (C1) [below] {$Lcd$};

\draw (A1) -- (A2) -- (B1) -- cycle;
\draw (A4) -- (A1) -- (B1) -- cycle;
\draw (A1) -- (A2) -- (C1) -- cycle;
\draw (A4) -- (A1) -- (C1) -- cycle;
\draw (A2) -- (A3) -- (B1) -- cycle;
\draw (A3) -- (A4) -- (B1) -- cycle;
\draw (A2) -- (A3) -- (C1) -- cycle;
\draw (A3) -- (A4) -- (C1) -- cycle;
\draw [fill opacity=0.7,fill=orange!80!black] (A1) -- (B1) -- (A3) -- (C1) -- cycle;
\end{tikzpicture} \quad 
    \begin{tikzpicture}[
    line join=bevel,
    z=-5.5, 
    ]
\coordinate (A1) at (0,0,-1);
\node at (A1) [above right] {$Lac$}; 
\coordinate (A2) at (-1,0,0);
\node at (A2) [left] {$Lad$}; 
\coordinate (A3) at (0,0,1);
\node at (A3) [below left] {$Lbd$}; 
\coordinate (A4) at (1,0,0);
\node at (A4) [right] {$Lbc$};
\coordinate (B1) at (0,1,0);
\node at (B1) [above] {$Lab$};
\coordinate (C1) at (0,-1,0);
\node at (C1) [below] {$Lcd$};

\draw (A1) -- (A2) -- (B1) -- cycle;
\draw (A4) -- (A1) -- (B1) -- cycle;
\draw (A1) -- (A2) -- (C1) -- cycle;
\draw (A4) -- (A1) -- (C1) -- cycle;
\draw (A2) -- (A3) -- (B1) -- cycle;
\draw (A3) -- (A4) -- (B1) -- cycle;
\draw (A2) -- (A3) -- (C1) -- cycle;
\draw (A3) -- (A4) -- (C1) -- cycle;
\draw [fill opacity=0.7,fill=purple!70!black] (A2) -- (B1) -- (A4) -- (C1) -- cycle;
\end{tikzpicture}
\caption{The octahedral face $F_{L; \{a,b,c,d\}}$ of $\Delta_{k,n}$ and its matroid subdivisions, called \emph{2-splits}. The picture on the right is taken by seeing the octahedron from below.}
\end{figure}
	
	\begin{proof}
		We need to prove that every positive linear combination $c_1 \pi_1 + \cdots + c_d \pi_d$ is a tropical Plücker vector. Actually, it is sufficient to prove that $\pi_1 + \cdots + \pi_d \in \Dr(k,n)$. In fact, the tropical Plücker vectors $\pi_i$ and $c_i \pi_i$ induce the same subdivision of $\Delta_{k,n}$, and since the subdvision induced by $\pi_i + \pi_j$ is the common refinement of the subdivision induced by $\pi_i$ and the subdivision induced by $\pi_j$, the fact that $\pi_i + \pi_j$ induces a matroid subdivision implies that $c_i \pi_i + c_j \pi_j$ induces a matroid subdivision, i.e.\ $c_i \pi_i + c_j \pi_j \in \Dr(k,n)$ for all $i, j \in [n]$.\\
        Now, suppose by contradiction that $\pi = \pi_1+\cdots +\pi_d$ is not in $\text{Dr}(k,n)$; this means that at least one three-term tropical Plücker relation is not satisfied, i.e.\ there exists $L,\{a,b,c,d\}$ such that the minimum among:
		\[
		\pi_{Lab} + \pi_{Lcd},\ \pi_{Lac} + \pi_{Lbd},\ \pi_{Lad} + \pi_{Lbc}
		\]
        is achieved exactly once. Hence, the octahedral face of $\Delta_{k,n}$ defined by $x_\ell = 1$ for all $\ell\in L$ and $x_a+x_b+x_c+x_d = 2$ is subdivided by two incompatible two-splits.  But each $\pi_1,\ldots, \pi_d$ induces a two-split, so there exists a pair $\pi_i,\pi_j$ such that $\pi_i+\pi_j$ induces a non-matroidal subdivision of that octahedron, i.e.\ $\pi_i + \pi_j \not \in \Dr(k,n)$, a contradiction.
	\end{proof}

	A direct consequence of the previous result is the following characterization of the cones of the Dressian, which is essential for computations.

	\begin{theorem}
		\label{lem: full dressian cliques}
		Let $R_{k,n}$ be the set of (primitive integer) generators of the rays of the Dressian $\Dr(k,n)$ modulo lineality. Given any maximal-by-inclusion subcollection $\{\pi_1,\ldots, \pi_d\}$ of $R_{k,n}$ consisting of pairwise compatible tropical Plücker vectors, i.e.\ $\pi_i + \pi_j \in \Dr(k,n)$ for all $i,j \in [d]$, then their conical hull: 
		\[
		\text{cone}(\pi_1, \ldots, \pi_d) = \left\{t_1\pi_1:\cdots + t_d\pi_d \in \R^{\binom{n}{k}}: t_1, \ldots, t_d\geq 0 \right\}
		\]
		is a maximal cone in $\Dr(k,n)$.  Moreover, any maximal cone of $\Dr(k,n)$ is characterized uniquely by such a collection.
	\end{theorem}
	\begin{proof}
     	By \Cref{lemma:computationfulldr}, we already know that a cone as above is a maximal cone of $\Dr(k,n)$. Vice-versa, given a maximal cone $C \subset \Dr(k,n)$, then (the primitive integer generators of) its rays are uniquely determined up to lineality, and any non-negative linear combination of any subset of them is a tropical Plücker vector in $C$.  In particular, they are pairwise compatible and we have our desired collection.
	\end{proof}
    
	The situation is analogous for chirotropical Dressians, except for the following result.
    
	\begin{lemma}
	\label{lemma:2outof3}
		Fix a chirotope $\chi$ and any octahedral face $F_{L,\{a,b,c,d\}}$ of $\Delta_{k,n}$.  Then, exactly two out of the three possible splits of $F_{L,\{a,b,c,d\}}$ are compatible with $\chi$.
	\end{lemma}
	\begin{proof}
		This is a direct translation of the chirotropical Plücker relation on the index set $(L; \{a,b,c,d\})$:
		\begin{eqnarray}
			\pi_{Lac} + \pi_{Lbd} & = & \min\{\pi_{Lab} + \pi_{Lcd},\pi_{Lad} + \pi_{Lbc}\}.
		\end{eqnarray}
		These subdivisions are determined by the hyperplanes $x_a+x_b=1$ and $x_a+x_d=1$ with $x_\ell = 1$ for all $\ell \in L$. 
	\end{proof}
	This result implies that non-negative linear combinations of pairwise compatible chirotropical Plücker vectors are again chirotropical Plücker vectors. A direct consequence is the following characterization of the maximal cones of the chirotropical Dressian, important for computations.
	\begin{theorem}
		\label{lem: full chirotropical dressian cliques}
		Let $\chi \in \{ \pm 1 \}^{\binom{n}{k}}$ be a realizable, uniform chirotope and denote by $R^\chi_{k,n}$ the set of (primitive integer) generators of rays of the chirotropical Dressian $\Dr^\chi(k,n)$. Given any maximal-by-inclusion subcollection $\{\pi_1,\ldots, \pi_d\}$ of $R^\chi_{k,n}$ consisting of pairwise compatible $\chi$-tropical Plücker vectors, i.e.\ $\pi_i + \pi_j \in \Dr^\chi(k,n)$ for all $i, j \in [d]$, then their conical hull
		\[
		\text{cone}(\pi_1, \ldots, \pi_d) = \left\{t_1\pi_1:\cdots + t_d\pi_d \in \R^{\binom{n}{k}}: t_1, \ldots, t_d\geq 0 \right\}
		\]
            is a maximal cone in $\Dr^\chi (k,n)$.  Moreover, any maximal cone of $\Dr^\chi (k,n)$ is characterized uniquely by such a collection.
	\end{theorem}

    \section{Computing Dressians and Chirotropical Dressians}
    \label{sec:computing}
    
    	This section is devoted to algorithms for computing Dressians and chirotropical Dressians. This is based on the theoretical background of \Cref{lem: full dressian cliques} and \Cref{lem: full chirotropical dressian cliques}, which characterize maximal cones of the Dressians and of the chirotropical Dressians respectively.

	First, recall that $x \in \R^{\binom{n}{k}}$ belongs to the Dressian $\Dr(k,n)$ provided that the minimum in
	\begin{equation}
	\label{eqn:min+min-dr}
		\min_{\alpha \in \text{supp}(g)} \{\alpha \cdot x \}
	\end{equation}
 	is attained at least twice, for every 3-term Plücker relation $g \in P_{k,n}$. The finiteness of this set of inequalities allows us to construct a subroutine $\text{SatisfyEqn}$, with input a point $x \in \R^{\binom{n}{k}}$ and output \texttt{True} if $x$ satisfies the inequalities of $\Dr(k,n)$ given in \Cref{eqn:min+min-dr}, \texttt{False} otherwise.

	\begin{algorithm}
		\caption{Computation of the maximal cones of $\Dr(k,n)$ from the rays of $\Dr(k,n)$.}
		\label{algorithm-dr}
		
		\SetKwInOut{Input}{Input}
		\SetKwInOut{Output}{Output}
		\SetKwInOut{Subroutines}{Subroutines}
		\SetKwComment{Comment}{/ }{ }
		
		\Input{- $R$, the list of rays of $\Dr(k,n)$ up to lineality;}
		\Subroutines{- $\text{SatisfyEqn}$, with input a point $x \in \R^{\binom{n}{k}}$ and output \texttt{True} if $x$ satisfies the inequalities of $\Dr(k,n)$ given in \Cref{eqn:min+min-dr}, \texttt{False} otherwise; \newline
			- MaximalCliques, with input a graph $G$ and output the list of its maximal cliques;
		}
		\Output{- The list of maximal cones of $\Dr(k,n)$. Each maximal cone is given by the list of its rays up to lineality.
		}
		$\text{CompatiblePairs} \gets \{\{r_1,r_2\} \subset R \mid \text{SatisfyEqn}(r_1 + r_2) = \texttt{True}\}$ \;
		$G \gets $ Graph with vertex set $R$ and edge set CompatiblePairs \;
		$\text{Facets} \gets \text{MaximalCliques}(G)$ \Comment*[r]{facets of $\Dr(k,n)$}
		\Return $\text{Facets}$
	\end{algorithm}

	\Cref{algorithm-dr} computes maximal cones of the Dressian $\Dr(k,n)$. The input data is the list of rays of the full Dressian, together with a subroutine which checks if a point is in the Dressian and a subroutine which computes the maximal cliques of a given graph\footnote{We recall that a maximal clique of a graph $G = (V,E)$ is a subset $X \subset V$ with the property that the induced subgraph $G[X] \subset G$ is a complete graph, and which is maximal by inclusion with respect to this property.}. The procedure starts by constructing the list of compatible pairs. These are pairs of rays of $\Dr(k,n)$ whose sum belongs still to $\Dr(k,n)$. This is equivalent to taking the vector associated to the common refinement of the matroid subdivisions induced by the two tropical Plücker vectors. Afterwards, the algorithm constructs a graph $G$ with vertex set the list of rays and edges the collection of compatible rays of $\Dr(k,n)$. The facets of $\Dr(k,n)$, i.e.\ the maximal cones, are determined by the maximal cliques of $G$. The fact that every facet of $\Dr(k,n)$ is of this form is due to \Cref{lem: full dressian cliques}. At the level of tropical Plücker vectors, these are collections of rays which are pairwise compatible and maximal with respect to inclusion.
	
	Using the output of \Cref{algorithm-dr}, it is possible to compute the lower-dimensional cones of $\Dr(k,n)$. These are determined by intersecting pairs of maximal cones (by taking the common collection of rays among their generators), then each pair of maximal cones with another maximal cone, and so on.
 
	As in the previous section, the situation for chirotropical Dressians is analogous. Fix a chirotope $\chi \in \{\pm 1\}^{\binom{n}{k}}$. By \Cref{eqn:pluecker}, the $\chi$-tropical Plücker vectors are determined by checking every equality of the form given in \Cref{eqn:min+min-}. This reduces to checking an equation of the form
	\begin{equation}
		\label{eqn:min+min-2}
		\min_{\alpha \in \text{supp}^+(g)} \{ \alpha \cdot x \} = \min_{\alpha \in \text{supp}^-(g)} \{ \alpha \cdot x \}   
	\end{equation}
	for any 3-term Plücker relation $g \in P_{k,n}$, where $\text{supp}^\pm(g)$ is the collection of monomials $m \in \text{supp}(g)$ such that $\text{sign}(m(\chi)) = \pm 1$. The finiteness of this condition, once one manages to implement a method to partition $\text{supp}(g) = \text{supp}^+(g) \cup \text{supp}^-(g)$, makes it easy to check computationally. In particular, it allows us to construct a subroutine $\text{SatisfyEqn}^\chi$, with input a point $x \in \R^{\binom{n}{k}}$ and output \texttt{True} if $x$ satisfies the inequalities of $\Dr^\chi(k,n)$ given in \Cref{eqn:min+min-2}, \texttt{False} otherwise.
	
	\begin{algorithm}[H]
		\caption{Computation of rays and maximal cones of $\Dr^\chi(k,n)$ from rays of $\Dr(k,n)$.}
		\label{algorithm-chi}
		
		\SetKwInOut{Input}{Input}
		\SetKwInOut{Output}{Output}
		\SetKwInOut{Subroutines}{Subroutines}
		\SetKwComment{Comment}{/ }{ }
		
		\Input{- $R$, the list of rays of $\Dr(k,n)$ up to lineality;}
		\Subroutines{- $\text{SatisfyEqn}^\chi$, with input a point $x \in \R^{\binom{n}{k}}$ and output \texttt{True} if $x$ satisfies the inequalities of $\Dr^\chi(k,n)$ given in \Cref{eqn:min+min-2}, \texttt{False} otherwise; \newline
			- MaximalCliques, with input a graph $G$ and output the list of its maximal cliques;
		}
		\Output{- The list of rays and the list of maximal cones of $\Dr^\chi(k,n)$. Each maximal cone is given by the list of its rays up to lineality.
		}
		$R^\chi \gets $ $\{ r \in R \mid \text{SatisfyEqn}^\chi(r) = \texttt{True}\}$ \Comment*[r]{rays of $\Dr^\chi(k,n)$}
		$\text{CompatiblePairs}^\chi \gets \{\{r_1,r_2\} \subset R^\chi \mid \text{SatisfyEqn}^\chi(r_1 + r_2) = \texttt{True}\}$ \;
		$G^\chi \gets $ Graph with vertex set $R^\chi$ and edge set $\text{CompatiblePairs}^\chi$ \;
		$\text{Facets}^\chi \gets \text{MaximalCliques}(G^\chi)$ \Comment*[r]{facets of $\Dr^\chi(k,n)$}
		\Return $R^\chi$, $\text{Facets}^\chi$
	\end{algorithm}
	
	\Cref{algorithm-chi} computes the rays and the maximal cones of the chirotropical Dressian $\Dr^\chi(k,n)$. The input data is the list of rays of the full Dressian, together with a subroutine which checks every equation of the chirotropical Dressian and a subroutine which computes the maximal cliques of a given graph. The procedure starts by selecting the rays of $\Dr^\chi(k,n)$ among the rays of $\Dr(k,n)$. This is done simply by checking which rays satisfy \Cref{eqn:min+min-2} for every 3-term Plücker relation $g \in P_{k,n}$. Then, it constructs a list of $\chi$-compatible pairs. These are pairs of rays of $\Dr^\chi(k,n)$ whose sum belongs still to $\Dr^\chi(k,n)$. This is equivalent to taking the vector associated to the common refinement of the matroid subdivisions induced by the two $\chi$-tropical rays. Afterwards, the algorithm constructs a graph $G^\chi$ with vertex set the list of chirotropical rays and edges the collection of $\chi$-compatible rays of $\Dr^\chi(k,n)$. The facets of $\Dr^\chi(k,n)$, i.e.\ the maximal cones, are determined by the maximal cliques of $G^\chi$. The fact that every facet of $\Dr^\chi(k,n)$ is of this form is due to \Cref{lem: full chirotropical dressian cliques}. At the level of $\chi$-tropical Plücker vectors, these are collections of chirotropical rays which are pairwise $\chi$-compatible and maximal with respect to inclusion.

    \begin{remark}\label{rem: higher codim faces}
        As in the case of the full Dressian, the output of \Cref{algorithm-chi} can be used to compute the lower-dimensional cones of $\Dr^\chi(k,n)$. These are determined by intersecting pairs of maximal cones (by taking the common collection of rays among their generators), then each pair of maximal cones with another maximal cone, and so on. In our computations we realized that actually, in order to get the lower dimensional cones, only pairwise intersections of maximal cones are needed. This is summarized in the next section in \Cref{thm:grassmannian2determined}.
    \end{remark}

	\section{Realizability and computation of $\textnormal{Dr}^\chi(3,n)$, $n = 6,7,8$}\label{sec: computations}
        \label{sec:3678}

        As already noted in \Cref{note:+}, the positive Dressian is among the chirotropical Dressians. It is well known that the positive Dressian is realizable, i.e.\ equal to the positive tropical Grassmannian \cite{speyerwilliamspos2020}. It is a natural question to understand when the chirotropical Dressian equals the chirotropical Grassmannian. We begin this section with the proof of our main realizability result, conjectured in \cite{CEZ2024A2}. The computations in this proof were performed using \Cref{algorithm-chi}, which we introduced in the previous section.
	
	\begin{proof} (\Cref{thm:realizability-intro})
		When $n=6$ it is well-known \cite{speyersturmfels2004} that the Dressian $\text{Dr}(3,6)$ is set theoretically equal to the tropical Grassmannian $\text{Trop }\G(3,6)$; $n=7,8$ require an argument.  For $n=7$ we considered the interior of any of the seven-dimensional Fano cones. They are obtained from the cone whose rays are among the standard basis $e_{ijk}$ of $\R^{\binom{n}{3}}$ and are labeled by 3-subsets $ijk$ corresponding to the seven nonbases of the Fano matroid:
		\[
		167, \ 246, \ 356, \ 237, \ 457, \ 125,\ 134.
		\]
		This cone and each relabeling of it under the $S_7$ action on $\R^{\binom{7}{3}}$ mark the difference between $\Trop \ \G(3,7)$ and $\Dr(3,7)$. We checked that every cone of this form is not compatible with any of the 11 isomorphism classes of chirotopes $\chi$ reported in \cite{FinschiCatalog}, even though any of its six-dimensional facets is compatible with some chirotope. Next, we apply the result of \cite{BENDLE2024} for $\text{Dr}(3,8)$: as for $n=7$, it suffices to check that only the six-dimensional faces of the extended Fano cones are in any given chirotropical Dressian. We took all permutations under the symmetric group $S_8$ of a point in the relative interior of the Fano cone and found that none of them are compatible with any of the 135 isomorphism classes of chirotopes $\chi$ reported in \cite{FinschiCatalog}.
	\end{proof}
    
        \begin{remark}
        \label{rmk:covering}
            In the proof of \Cref{thm:realizability-intro}, we also checked that, for $n = 6, 7$, the tropical Grassmannian $\Trop \ \G(3,n)$ is covered by chirotropical Grassmannians $\Trop^\chi \G(3,n)$. We considered the isomorphism classes reported in \cite{FinschiCatalog} and took their relabelings via the symmetric group $S_n$ and their orbit under the action of the torus $\{ \pm 1 \}^n$. We obtained all rank three chirotopes on $n$ elements in this way. Then, by means of \Cref{algorithm-chi}, we checked that each pair of rays $\pi_i, \pi_j$ of the tropical Grassmannian such that $\pi_i + \pi_j \in \Trop \ \G(3,n)$ satisfies also that $\pi_i + \pi_j \in \Trop^\chi \G(3,n)$ for some chirotope $\chi$.
        \end{remark}
	
    We now compute all chirotropical moduli spaces  $\text{Trop}^\chi X(3,n)$ for $n=6,7,8$. In these cases, by \Cref{thm:realizability-intro}, we know that the chirotropical Grassmannian is equal to the chirotropical Dressian. Our implementation of \Cref{algorithm-dr} and \Cref{algorithm-chi} is in SageMath \cite{sagemath} and it is designed for any Dressian $\Dr(k,n)$ in characteristic zero.
	
	Through the rest of the section, for sake of compactness, we will denote a rank 3 chirotope by $\chi \in \{\pm\}^{\binom{n}{3}}$. We use the standard lexicographic order of the 3-subsets:
	\[
	123 < 124 < \cdots < 12n < \cdots < (n-2)(n-1)n.
	\]
	We show the results of our computations.
	\begin{theorem}[Case $(3,6)$]
		\label{thrm:36}
		For any of the 4 isomorphism classes of realizable uniform chirotopes $\chi \in \{\pm\}^{\binom{6}{3}}$ in \cite{FinschiCatalog}, the chirotropical Dressian $\Dr^\chi(3,6)$ modulo lineality is a pure $4$-dimensional polyhedral fan and it is equal to the chirotropical moduli space $\Trop^\chi X(3,6)$. These polyhedral fans have f-vectors:
		\begin{center}
			\begin{tabular}{|c|c|c|}
				\hline
				\# & $\chi$ & f-vector\\
				\hline
				1 & $(++++++++--+++++-++++)$ & $(15, 60, 90, 45)$ \\
				\hline
				2 & $(++++++++++++++++++--)$ & $(15, 60, 89, 44)$ \\
				\hline
				3 & $(+++++++++++++++++++-)$ & $(14, 55, 82, 41)$ \\
				\hline
				4 & $(++++++++++++++++++++)$ & $(16, 66, 98, 48)$ \\
				\hline
			\end{tabular}
		\end{center}
	\end{theorem}
	
	\noindent From now on, we will denote a chirotope $\chi \in \{ \pm \}^{\binom{n}{3}}$ as the vector indexing $3$-subsets $ijk \subset [n]$ such that $\chi_{ijk} = -$. As an example, the chirotope:
	\[
	(++++++++++++++++++++++++++++++++++-) \in \{ \pm \}^{\binom{7}{3}}
	\]
	will be denoted by $(567)$. We reserve the notation $+$ for the totally positively oriented chirotope $(+\cdots+)$.
	
	\begin{theorem}[Case $(3,7)$]
		\label{thrm:37}
		For any of the 11 isomorphism classes of realizable uniform chirotopes $\chi \in \{\pm\}^{\binom{7}{3}}$ in \cite{FinschiCatalog}, the chirotropical Dressian $\Dr^\chi(3,7)$ modulo lineality is a pure $6$-dimensional polyhedral fan and it is equal to the chirotropical moduli space $\Trop^\chi X(3,7)$. These polyhedral fans have f-vectors:
		\begin{center}
			\begin{tabular}{|c|c|c|}
				\hline
				\# & $\chi$ & f-vector\\
				\hline
				1 & $(356,456,457,467)$ & $(30, 244, 864, 1513, 1287, 424)$ \\
				\hline
				2 & $(267,357,367,456,457,467,567)$ & $(31, 252, 892, 1565, 1335, 441)$ \\
				\hline
				3 & $(345,467,567)$ & $(28, 222, 781, 1373, 1179, 393)$ \\
				\hline
				4 & $(356,357,456,457)$ & $(39, 342, 1224, 2109, 1746, 558)$ \\
				\hline
				5 & $(356,456,457)$ & $(35, 298, 1073, 1885, 1597, 522)$ \\
				\hline
				6 & $(367,456,457,467,567)$ & $(34, 291, 1050, 1844, 1560, 509)$ \\
				\hline
				7 & $(367,457,467,567)$ & $(36, 311, 1125, 1974, 1665, 541)$ \\
				\hline
				8 & $(457,467,567)$ & $(30, 248, 891, 1577, 1351, 447)$ \\
				\hline
				9 & $(467,567)$ & $(37, 325, 1181, 2070, 1740, 563)$ \\
				\hline
				10 & $(567)$ & $(34, 296, 1084, 1922, 1634, 534)$ \\
				\hline
				11 & + & $(42, 392, 1463, 2583, 2163, 693)$ \\
				\hline
			\end{tabular}
		\end{center}
	\end{theorem}
	
	\begin{theorem}[Case $(3,8)$]
		\label{thrm:38}
		For any of the 135 isomorphism classes of realizable uniform chirotopes $\chi \in \{\pm\}^{\binom{8}{3}}$ in \cite{FinschiCatalog}, the chirotropical Dressian $\Dr^\chi(3,8)$ modulo lineality is a pure $8$-dimensional polyhedral fan and it is equal to the chirotropical moduli space $\Trop^\chi X(3,8)$. The f-vectors of these polyhedral fans can be found in \cite{datachirotropicalization20241}.
	\end{theorem}
	
	\begin{remark}
		The f-vectors of the positive parts, i.e.\ corresponding to the totally positive chirotope $+$, agree with those of the totally positive Tropical Grassmannians already known in the literature. For example, for the cases $(3,6)$ and $(3,7)$, they agree with the results of Speyer and Williams \cite{speyerwilliams2004}, and for the case $(3,8)$ they are equal to the f-vectors obtained by Bendle, Böhm, Ren and Schröter \cite{BENDLE2024}. The number of maximal cones, that is, the last entries in the f-vectors, coincide with the results reported in \cite{CEZ2024A2} for the numbers of Generalized Feynman Diagrams (GFD) for all $4,11,135$ types of chirotopal tropical Grassmannians.
	\end{remark}
	
	As already discussed in \Cref{rem: higher codim faces}, we verified that, in order to get the lower dimensional cones from the maximal cones, it suffices to intersect only pairs of maximal cones. This gives us the following result.
	
	\begin{theorem}
		\label{thm:grassmannian2determined}
        For $n=6,7,8$, for any chirotope $\chi$ of rank 3, the chirotropical moduli space $\Trop^\chi X(3,n)$ has the following property: every non-maximal cone can be expressed as the intersection of two maximal cones.
	\end{theorem}
 
    \section{The chirotopal configuration spaces $X^\chi (3,6)$ are polytopal}
    \label{sec:36polytopal}
    
	In this section, we present a key result about chirotopal configuration spaces $X^\chi(3,6)$. In \Cref{rmk:covering}, we explained how to obtain all 372 rank 3 uniform realizable chirotopes on 6 elements from representatives of their isomorphism classes; for each of these, we exhibit a birational map $\mathbb{R}^4 \dashrightarrow X(3,6)$ which restricts to a diffeomorphism $\mathbb{R}^4_{>0} \rightarrow X^\chi(3,6)$, and we propose a canonical form in the sense of positive geometry \cite{ABL}. 
    
	\begin{theorem}
		Each chirotopal configuration space $X^\chi(3,6)$ is diffeomorphic to a polytope.
	\end{theorem}
    In the course of the proof, for each chirotope we propose a canonical differential form, which will imply that $X(3,6)$ is tiled with positive geometries.  
    The derivation of the parameterizations and canonical differential forms borrows techniques from positive del Pezzo moduli spaces \cite{EAPSY} and is somewhat beyond the scope of the paper to explain in detail; we refer to it for context and motivation.
    
	\begin{proof}
        We exhibit parameterizations for one representative from each of the four isomorphism classes of chirotopes.  In each case, we checked that the normal fan of the Newton polytope of the product of all irreducible polynomials occurring in the minors is isomorphic to the fan which we computed in \Cref{thrm:36}. 

        The totally positively oriented chirotope $+$ of type 4 in \Cref{thrm:36} admits the following parameterization:
		$$
\begin{bmatrix}
 1 & 0 & 0 & y_1 y_3 & y_1 y_3+y_1 y_4+y_2 y_4 & y_3 y_1+y_4 y_1+y_1+y_2+y_2 y_4+1 \\
 0 & 1 & 0 & -y_3 & -y_3-y_4 & -y_3-y_4-1 \\
 0 & 0 & 1 & 1 & 1 & 1 \\
\end{bmatrix}.
		$$
        We can solve explicitly for the parameters $y_i$:
        $$\left(\frac{p_{145} p_{156} p_{234}}{p_{125} p_{134} p_{456}},\frac{p_{124} p_{156} p_{345}}{p_{125} p_{134} p_{456}},\frac{p_{125} p_{126} p_{134}}{p_{123} p_{124} p_{156}},\frac{p_{126} p_{145}}{p_{124} p_{156}}\right) = (y_1,y_2,y_3,y_4).$$
		The type 3 chirotope in \Cref{thrm:36} with $\chi_{456}=-1$ admits the following parameterization:
		$$
		\begin{bmatrix}
			1 & 0 & 0 & 1 & \frac{y_1+1}{y_1} & \frac{y_1 y_3+y_3+1}{y_1 y_3} \\
			0 & 1 & 0 & -1 & -\frac{\left(y_1+1\right) \left(y_2 y_4+y_4+1\right)}{y_2 y_4 y_1+y_4 y_1+y_1+y_4+1} & -\frac{\left(y_1 y_3 y_2+y_3 y_2+y_2+y_1 y_3+y_3\right) \left(y_2 y_4+y_4+1\right)}{\left(y_2+1\right) y_3 \left(y_2 y_4 y_1+y_4 y_1+y_1+y_4+1\right)} \\
			0 & 0 & 1 & 1 & 1 & 1 \\
		\end{bmatrix}
		.$$
		We can solve explicitly for the parameters $y_{i}$:
		$$\left(\frac{p_{125} p_{234}}{p_{123} p_{245}},\frac{p_{156} p_{235}}{p_{125} p_{356}},\frac{p_{126} p_{245}}{p_{124} p_{256}},-\frac{p_{145} p_{356}}{p_{135} p_{456}}\right) = (y_1,y_2,y_3,y_4).$$
            The type 2 chirotope in \Cref{thrm:36} with
		$$\chi _{134}=\chi _{135}=\chi _{136}=\chi _{235}=\chi _{236}=\chi _{245}=\chi _{246}=\chi _{256} = -1$$
		admits the following parameterization:
		$$
		\begin{bmatrix}
			1 & 0 & 0 & 1 & -\frac{1}{y_1} & -\frac{y_1 y_2 y_3+y_2 y_3+y_1 y_2 y_4 y_3+y_2 y_4 y_3+y_4 y_3+y_3+y_1 y_2 y_4+y_2 y_4+y_4+1}{y_1 \left(y_3+1\right) \left(y_1 y_2 y_4+y_2 y_4+y_4+1\right)} \\
			0 & 1 & 0 & 1 & \frac{y_2}{y_2+1} & \frac{y_2}{\left(y_2+1\right) \left(y_3+1\right)} \\
			0 & 0 & 1 & 1 & 1 & 1 \\
		\end{bmatrix}.
		$$
		We can solve for the parameters $y_i$:
		$$\left(-\frac{p_{125} p_{234}}{p_{124} p_{235}},-\frac{p_{124} p_{135}}{p_{123} p_{145}},-\frac{p_{123} p_{156}}{p_{125} p_{136}},\frac{p_{235} p_{456}}{p_{256} p_{345}}\right) = (y_1,y_2,y_3,y_4).$$
		The type 1 chirotope in \Cref{thrm:36} with:
		$$\chi _{134}=\chi _{135}=\chi _{145}=\chi _{235}=\chi _{245}=\chi _{346}=\chi _{356}=-1$$

		admits the following parameterization:
		$$
		\begin{bmatrix}
			1 & 0 & 0 & 1 & -\frac{1}{y_3} & \frac{y_2 y_3+y_3+1}{y_2 y_3} \\
			0 & 1 & 0 & 1 & \frac{\left(y_1 y_2+y_2+1\right) \left(y_4+1\right)}{y_1 y_2 y_4+y_2 y_4+y_4+1} & -\frac{y_1 \left(y_2 y_3+y_3+1\right) \left(y_4+1\right)}{y_2 y_3+y_2 y_4 y_3+y_4 y_3+y_3+y_1 y_2 y_4+y_2 y_4+y_4+1} \\
			0 & 0 & 1 & 1 & 1 & 1 \\
		\end{bmatrix}.
		$$
		We can solve for the $y_i$ parameters:
        \[
            \left(-\frac{p_{124} p_{136} p_{256} p_{345}}{p_{126} p_{134} p_{245} p_{356}},-\frac{p_{126} p_{245}}{p_{125} p_{246}},-\frac{p_{125} p_{234}}{p_{124} p_{235}},-\frac{p_{125} p_{134} p_{236} p_{456}}{p_{123} p_{145} p_{256} p_{346}}\right)\\
            = (y_1,y_2,y_3,y_4).
        \]
            For each chirotope $\chi$, the canonical form on $X^\chi(3,6)$ is given by the wedge of dlogs of these four cross-ratios:
		$$\Omega^\chi = d\log(y_1)\wedge d\log(y_2)\wedge d\log(y_3) \wedge d\log(y_4).$$
        which we abbreviate by $\Omega_1,\Omega_2,\Omega_3,\Omega_4$.
        These evaluate to the following rational functions:
        \begin{itemize}
        \item Type 4 chirotope:
        $$\Omega_4 = \frac{1}{p_{123} p_{234} p_{345} p_{456} p_{561} p_{612}}d\mathbf{y},$$
        \item Type 3 chirotope:
        $$\Omega_3=\frac{1}{p_{123} p_{126} p_{145} p_{234} p_{356} p_{456}}d\mathbf{y},$$
        \item Type 2 chirotope: $$\Omega_2 = 
        \frac{p_{245}}{p_{124} p_{136} p_{145} p_{234} p_{235} p_{256} p_{456}}
        d\mathbf{y},$$
        \item Type 1 chirotope:
        $$ \Omega_1= \frac{p_{123} p_{345} p_{156} p_{246}-p_{234} p_{456} p_{126} p_{135}}{p_{125} p_{126} p_{134} p_{136} p_{145} p_{234} p_{235} p_{246} p_{356} p_{456}}
        d\mathbf{y}.$$
        \end{itemize}
     \end{proof}
    The rational functions multiplying the volume form $d\mathbf{y} = dy_1\wedge dy_2 \wedge dy_3 \wedge dy_4$ in these four canonical differential forms that we propose coincide (up to relabeling) with Equations 3.15, 3.16, 3.17 and 3.26 in \cite{CEZ2024A4B}.  These, in turn, appeared in \cite[Section 2]{CEZ2024A2}, relating the CEGM integral to the cone-by-cone Laplace transform of the chirotropical moduli spaces $\text{Trop}^\chi X(3,6)$ which we have computed in this work.

        Given such diffeomorphisms as above, one can directly construct a parameterization of the chirotropical Grassmannian and a Global Schwinger Parameterization \cite{CE2024} of the corresponding generalized biadjoint scalar amplitude.  Our impression is that a similar formulation of parameterizations  can be done in the case of $X^\chi(3,7)$ and $X^\chi(3,8)$. Motivated by this and by \Cref{thm:realizability-intro}, we propose the following conjecture.

        \begin{conjecture}
        If there exists a birational map $\R^{(k-1)(n-k-1)} \dashrightarrow X(k,n)$ which restricts to a diffeomorphism $\R_{>0}^{(k-1)(n-k-1)} \to X^\chi(k,n)$, then the equality of sets $\Trop^\chi \G(k,n) = \Dr^\chi(k,n)$ holds true.
	\end{conjecture}
	
	\section{Availability of the code}
        \label{sec:code}
	
	All the code and the results of the computations, stored as SageMath objects, are available in a Zenodo page at \cite{datachirotropicalization20241}. The material in the page includes:
	\begin{itemize}
		\item The implementation of \Cref{algorithm-chi} in SageMath;
		\item The implementation of an algorithm to generate all the Plücker relations in SageMath;
		\item Plücker relations and 3-term Plücker relations for the cases $(3,6)$, $(3,7)$ and $(3,8)$;
		\item The list of chirotope isomorphism classes in \cite{FinschiCatalog} for the cases $(3,6)$, $(3,7)$ and $(3,8)$;
		\item List of rays of the chirotropical Dressians (equal to the corresponding chirotropical Grassmannians) modulo lineality in the cases $(3,6)$, $(3,7)$ and $(3,8)$. We gratefully acknowledge the authors of \cite{BENDLE2024} for making their data on the tropical Grassmanniann $\Trop \ \G(3,8)$ publicly available;
		\item All the chirotropical Dressians (equal to the chirotropical Grassmannians) in the cases $(3,6)$, $(3,7)$ and $(3,8)$. A face of dimension $i$ is stored as a list of at least $i$ rays among the rays of the original tropicalization whose positive hull is equal to that face; for any $n = 6,7,8$, the full chirotropical Dressian $\Dr^\chi(3,n)$ is saved as a Python dictionary $d^\chi$ with keys $i = 1, \ldots, 2(n-4)$ such that $d^\chi[i]$ is the list of all faces of dimension $i$ for all $i \in [2(n-4)]$;
		\item A list of the f-vectors of all chirotropical Dressians $\Dr^\chi(3,8)$. The user can either choose a chirotope and obtain the f-vector, or choose a chirotope from the list of chirotopes and obtain the f-vectors as the element with the same index in the list of all f-vectors.
	\end{itemize}

	\section{Discussions and Future Work}
        \label{sec:futurework}
	
	It is an important problem to generalize our work in the context of the full real tropicalization of the Grassmannian using Puiseux series, see for example \cite{BENDLE2024}.  Some remaining questions are the following: what is the fan structure on a chirotropical Grassmannian?  Is the real tropical Grassmannian covered by chirotropical Grassmannians? The analogous question was confirmed in \cite{EAPSY} for the moduli space of del Pezzo surfaces $Y(3,6)$.
	
	Our work opens many questions. One of these concerns the realizability of the chirotropical Dressian. We begin with an example which shows that for rank four chirotopes, in general the chirotropical Dressian is not equal to the chirotropical Grassmannian.  The same question for rank three chirotopes remains open.
	
	In \cite{CEZ2024A2}, rank $k$ chirotopes were encoded by $\binom{n}{k-2}$-element collections of rank two chirotopes subject to certain compatibility conditions, and were called Generalized Color Orders (GCOs); now rank two chirotopes on $n$ labels modulo reorientation are well-known to be equivalent to dihedral orders on $[n]$, see for example \cite{bjorner}.  In what follows, using data from \cite{CEZ2024A2}, we present a rank four chirotope $\chi$ on eight elements, such that $\text{Dr}^\chi(4,8)$ contains a nonrealizable cone.  In the table which follows, each 6-tuple is considered up to cyclic permutation and reflection.  
	$$
	\begin{array}{cccccccc}
		& 345687 & 245867 & 238567 & 238476 & 257438 & 234658 & 263547 \\
		345687 &  & 148567 & 138576 & 167438 & 154738 & 136458 & 165347 \\
		245867 & 148567 &  & 128756 & 164728 & 127458 & 126548 & 152476 \\
		238567 & 138576 & 128756 &  & 127368 & 172358 & 162538 & 123765 \\
		238476 & 167438 & 164728 & 127368 &  & 123487 & 124386 & 132764 \\
		257438 & 154738 & 127458 & 172358 & 123487 &  & 142385 & 127543 \\
		234658 & 136458 & 126548 & 162538 & 124386 & 142385 &  & 123456 \\
		263547 & 165347 & 152476 & 123765 & 132764 & 127543 & 123456 &  \\
	\end{array}
	$$
	This is equivalent to the chirotope:
	$$
	\begin{array}{cccccccccccccc}
		1 & 1 & 1 & 1 & 1 & 1 & 1 & 1 & 1 & 1 & 1 & 1 & 1 & 1 \\
		-1 & 1 & 1 & 1 & 1 & 1 & 1 & 1 & 1 & -1 & -1 & 1 & 1 & -1 \\
		1 & -1 & -1 & -1 & -1 & -1 & -1 & 1 & 1 & 1 & 1 & 1 & 1 & -1 \\
		1 & -1 & -1 & 1 & 1 & -1 & -1 & -1 & -1 & -1 & -1 & -1 & -1 & 1 \\
		-1 & -1 & -1 & -1 & -1 & -1 & -1 & -1 & -1 & -1 & -1 & -1 & -1 & -1 \\
	\end{array}
	$$
	Here the rows are ordered lexicographically, starting with
    \[ \chi_{1234},\chi_{1278},\chi_{1467},\chi_{2367},\chi_{3457},
    \]
    respectively. We leave it to the reader to compute the simplicial cells in the (any) generic hyperplane arrangement in $\mathbb{P}^3$ determined by this GCO; the result of that computation is the following list of 20 tetrahedra:
	$$
	\begin{array}{ccccc}
		1234 & 1237 & 1256 & 1268 & 1278 \\
		1358 & 1368 & 1458 & 1467 & 1567 \\
		2348 & 2358 & 2367 & 2457 & 2467 \\
		3456 & 3457 & 3478 & 4568 & 5678 \\
	\end{array}.
	$$
	
	One can verify that each canonical basis vector $e_J \in \text{Dr}^\chi(4,8)$ for each $J = j_1j_2j_3j_4$ in the table, so we can tabulate maximal cliques and thereby produce cones of chirotropical Plücker vectors.  That calculation reveals that $\text{Dr}^\chi(4,8)$ contains a cone of dimension at least 12 which contains rays generated by the following 12 linearly independent vectors:
	$$
	\begin{array}{cccccc}
		e_{1234} & e_{1256} & e_{1278} & e_{1368} & e_{1458} & e_{1467} \\
		e_{2358} & e_{2367} & e_{2457} & e_{3456} & e_{3478} & e_{5678} \\
	\end{array}.
	$$
	But the dimension of $\text{Trop }X(4,8)$ is 9, hence the relative interior of this 12-dimensional cone cannot be realizable.

        Finally, another important aspect which may be further investigated concerns the generalization of the chirotropical Grassmannian to arbitrary (in principle, non-realizable and non-uniform) chirotopes, in the sense of \cite{bjorner}.
	
	\section{Acknowledgements}
    We thank Freddy Cachazo, Matteo Gallet, Alessandro Oneto, David Speyer and Yong Zhang for fruitful discussions, and Yassine El Maazouz and Lakshmi Ramesh for helpful comments on the manuscript.  We are particularly grateful to Bernd Sturmfels for encouragement and support during the early stages of the project. This started when D.\ A.\ was a visiting undergraduate at the Max Planck Institute for Mathematics in the Sciences in Leipzig, for the traineeship of his master's degree course at the University of Trieste. D.\ A.\ is grateful to Rainer Sinn for support and to the institute for hospitality during the visiting period. D.\ A.\ is a member of GNSAGA (INdAM) and was supported by the University College of Excellence "Luciano Fonda" during the traineeship. N.\ E.\ was funded by the European Union (ERC, UNIVERSE PLUS, 101118787). Views and opinions expressed are however those of the author(s) only and do not necessarily reflect those of the European Union or the European Research Council Executive Agency. Neither the European Union nor the granting authority can be held responsible for them.
	
	\begin{small}
		
	\end{small}
	
\end{document}